\documentclass[12pt,reqno]{amsart}

\usepackage{amssymb}
\usepackage{amsmath}
\usepackage{latexsym}
\usepackage{amsthm}
\usepackage{epsfig}
\usepackage{color}
\usepackage{comment}
\usepackage{amsfonts,amssymb,mathrsfs,amscd}
\usepackage{enumitem}
\usepackage{etoolbox}
\newtheorem{proposition}{Proposition}[section]
\newtheorem{theorem}[proposition]{Theorem}
\newtheorem{corollary}[proposition]{Corollary}

\theoremstyle{definition}
\newtheorem{definition}[proposition]{Definition}

\newtheorem{remark}[proposition]{Remark}

\numberwithin{equation}{section}

\def\R{\mathbb R}
\def\eb{\varepsilon}
\def\({\left(}
\def\){\right)}
\def\divv{\operatorname{div}}

\def\Dx{\Delta_x}
\def\Nx{\nabla_x}
\def\Dt{\partial_t}

\def\Bbb{\mathbb}
\def\Cal{\mathcal}
\begin{document}
 \title[Attractors for 3D damped Euler equations]{Trajectory attractors for 3D damped Euler equations and their approximation}
\author[A. Ilyin, A. Kostianko, and S. Zelik] {Alexei Ilyin${}^1$, Anna Kostianko${}^{3,4}$,
and Sergey Zelik${}^{1,2,3}$}



\thanks{This work was supported by Moscow Center for Fundamental and Applied Mathematics,
Agreement with the Ministry of Science and Higher Education of the Russian Federation,
No. 075-15-2019-1623 and by the Russian Science Foundation grant No.19-71-30004 (sections 2-4).
The second  author was  partially supported  by the Leverhulme grant No. RPG-2021-072 (United Kingdom).}

\email{ilyin@keldysh.ru}
\email{a.kostianko@imperial.ac.uk}
\email{s.zelik@surrey.ac.uk}
\address{${}^1$ Keldysh Institute of Applied Mathematics, Moscow, Russia}
\address{${}^2$ University of Surrey, Department of Mathematics, Guildford, GU2 7XH, United Kingdom.}
\address{${}^3$ \phantom{e}School of Mathematics and Statistics, Lanzhou University, Lanzhou\\ 730000,
P.R. China}
\address{${}^4$  Imperial College, London SW7 2AZ, United Kingdom.}

\begin{abstract}
We study the  global attractors for the damped 3D Euler--Bardina
equations with the regularization parameter $\alpha>0$ and Ekman
damping coefficient $\gamma>0$ endowed with periodic boundary
conditions as well as their damped Euler limit $\alpha\to0$. We
prove that despite the possible non-uniqueness of solutions of the
limit Euler system and even the non-existence of such solutions in
the distributional sense, the limit dynamics of the corresponding
dissipative solutions introduced by P.\,Lions can be described in
terms of attractors of  the properly constructed trajectory
dynamical system. Moreover, the convergence of the attractors $\Cal
A(\alpha)$ of the regularized system to the limit trajectory
attractor $\Cal A(0)$  as $\alpha\to0$ is also established in terms
of the upper semicontinuity in the properly defined functional
space.
\end{abstract}

\maketitle
\small\tableofcontents

\setcounter{equation}{0}

\section{Introduction}\label{sec0}
Being the central mathematical model in hydrodynamics, the
Navier--Stokes and Euler equations permanently remain in the focus
of both the analysis of PDEs and the theory of infinite dimensional
dynamical systems and their attractors, see
\cite{B-V,Ch-V-book,Fef06,FMRT,F95,Lad,L34,lions,tao,T,T95} and the
references therein for more details. Most studied is the 2D case
where reasonable results on the global well-posedness and
regularity of solutions as well as the results on the  existence of
global attractors and their dimension are available. However, the
global well-posedness  in
 the 3D case remains a mystery and even listed by the Clay institute of mathematics as one of the Millennium problems. This mystery inspires a comprehensive study of various modifications/regularizations of the initial Navier-Stokes/Euler equations (such as Leray-$\alpha$ model, hyperviscous Navier-Stokes equations, regularizations via $p$-Laplacian, etc.), many of which have a strong physical background and are of independent interest, see e.g. \cite{Camassa,HLT10,L34,Lopes,OT07} and the references therein.
\par
In the present paper we shall be dealing
 with the following regularized 3D damped Euler system:
%$$
\begin{equation}\label{DEalpha}
\left\{
  \begin{array}{ll}
    \partial_t u+(\bar u,\nabla_x)\bar u+\gamma u+\nabla_x p=g,\  \  \\
    \operatorname{div} \bar u=0,\quad u(0)=u_0.
  \end{array}
\right.
\end{equation}
%$$
with  forcing $g$ and  Ekman damping term $\gamma u$,
$\gamma>0$ endowed with periodic boundary conditions ($x\in\Bbb T:=[-\pi,\pi]^3$). The damping term $\gamma  u$ makes the system dissipative and
is important in  various geophysical models~\cite{Ped}. Here and below
 $\bar u$ is a smoothed (filtered)
vector field related with the initial velocity field $u$ by means
of the solution of the Stokes problem
%$$
\begin{equation}\label{0.bar}
u=\bar u-\alpha\Dx \bar u,\ \ \divv\bar u=0,
\end{equation}
%$$
where $\alpha>0$ is a given small parameter. In other words,
$$
\bar u=(1-\alpha\Dx)^{-1}u.
$$
\par
System \eqref{DEalpha}, \eqref{0.bar} (at least in the conservative case $\gamma=0$) is
often referred to as the  simplified Bardina subgrid
scale model of turbulence, see \cite{BFR80,Bardina,Lay} for the derivation
of the model and further discussion, so in this paper
we  shall be calling  \eqref{DEalpha}  the damped Euler--Bardina equations.
We also mention that rewriting \eqref{DEalpha} in terms of the variable $\bar u$ gives
%$$
\begin{equation}\label{0.KV}
\Dt\bar u-\alpha\Dt\Dx\bar u+(\bar u,\Nx)\bar u+\gamma(\bar u-\alpha\Dx \bar u)+\Nx p=g
\end{equation}
%$$
which is a damped version of the so-called Navier--Stokes--Voight equations arising
in the theory of viscoelastic fluids, see \cite{Titi-Varga,Osk} for the details.
In the sequel, we will mainly use the equivalent equations \eqref{0.KV}
instead of \eqref{DEalpha}.
\par
Our main interest in the present paper is to study the limit $\alpha\to0$
in terms of the attractors $\Cal A(\alpha)$ of the corresponding regularized
equations \eqref{0.KV}. We recall that, in contrast to the classical
Navier--Stokes approximations, the global well-posedness of Bardina--Euler
equations is well-known, see \cite{Titi-Varga,Lay,IKZ} for more details. Moreover, it is also known that, if $g\in L^2(\Bbb T^3)$ and $\alpha,\gamma>0$, problem \eqref{0.KV} possesses a global attractor $\Cal A(\alpha)$ of finite fractal dimension in the phase space $\Cal H_\alpha:=H^1(\Bbb T^3)\cap\{\divv \bar u=0\}$ endowed with the norm
$$
\|u\|_{\Cal H_\alpha}^2=\|u\|^2_{L^2}+\alpha\|\Nx u\|^2_{L^2}.
$$
The fractal dimension of this attractor possesses the following explicit estimate
%$$
\begin{equation}\label{0.est1}
\dim_F\mathscr A\le\frac{1}{12\pi}\frac{\| g\|_{L^2}^2}{\alpha^{5/2}\gamma^4}\,.
\end{equation}
%$$
and this estimate is sharp with respect to $\alpha$ and $\gamma$ in
the sense that the lower bounds of the same order are attained on a
family of specially constructed Kolmogorov flows, see \cite{IKZ,IKZ1}
for more details (see also \cite{IZLap70,IKZ1, IT1,IMT} for the
analogous results for 2D case).
\par
We however note that the above results do not give much information
about the behavior of $\Cal A(\alpha)$ as $\alpha\to0$ as well as about
the limit attractor $\Cal A(0)$ because of the dependence of the above
mentioned Kolmogorov flows on $\alpha$, see \cite{IKZ} for the discussion.
In particular, even in the relatively simple 2D case, the question
about the finite-dimensionality of the limit attractor remains completely open.
Note also that, due to the simplified structure of vorticity equations in 2D, we
have the uniform as $\alpha\to0$ $H^1$-estimate for the solutions of \eqref{0.KV}
and even can establish the uniqueness of slightly more regular solutions for
the limit Euler equations using the Yudovich technique (see \cite{Ju}).
In particular, this technique gives us the uniqueness on the attractor
$\Cal A(0)$ as well as the attraction to it in a strong topology of $H^1$,
see \cite{I,CIZ,CVZ,CZ,CIZ1} for more details.
\par
In contrast to 2D case, the situation in much more complicated in 3D. Indeed,
the presence of the so-called vorticity stretching term in the vorticity
equations prevents us from obtaining good estimates of the $H^1$-norm for
the limit Euler equations, so we have only the control of the $L^2$-norm
of the solution which comes from the basic energy estimate. But, unfortunately,
this control is not only  insufficient for the uniqueness, but even
the existence of weak solutions in the sense of distributions becomes
non-trivial.  For this reason, we have to use the notion of so-called
dissipative solutions introduced by P.\,Lions in order to verify the
solvability of the limit Euler equations, see \cite{PLio} and the
references therein. Note also that the distributional solutions for the classical
3D Euler equations can be constructed using the geometric integration
technique, see \cite{Wie} and references therein, but, in contrast
to dissipative solutions, these type of solutions usually do not satisfy the
energy inequality in a reasonable form  and cannot be obtained as the limit of
the corresponding solutions of the regularized system, so they look  not
very interesting from the point of view of attractors.
\par
However, there is one more problem here. Namely, the class of dissipative
solutions of the limit Euler equations is crucially not invariant with
respect to time shifts which prevents us to use the standard technique
of trajectory attractors at least in the straightforward way.  In order
to overcome this problem, we introduce a wider (than the dissipative ones)
class  of solutions of the limit Euler equations
$$
\Cal K^+_0\subset \Theta_+:=L^\infty_{w^*,loc}(\R_+,\Cal H_0)\cap W^{1,\infty}_{w^*,loc}(\R_+,H^{-3}),
$$
see sections \S\ref{S3} and \S\ref{S4} for more details. Roughly speaking,
the trajectory phase space $\Cal K_0^+$ consists of all limits of solutions
$\bar u_{\alpha_n}$ of the regularized Bardina--Euler equations \eqref{0.KV}
as $\alpha_n\to0$ in the topology of $\Theta_+$. In this case,
the semigroup of time-shifts $T(h):\Theta_+\to\Theta_+$, $h\ge0$,
will act on $\Cal K_0^+$ and we may define the trajectory dynamical
system $(\Cal K_+,T(h))$ associated with the limit Euler equation and
may speak about the global attractor $\Cal A_{tr}(0)$ of this
DS (=trajectory attractor of equation \eqref{0.KV} with $\alpha=0$).
\par
Note that the space $\Cal K^+_0$ is not empty and contains all
dissipative solutions of the limit Euler equation, but a priori may
be larger. This may be connected with the well known fact that the
weak limits of solutions of 3D Euler equations may not satisfy Euler
equations, see \cite{DiM} for more details.
\par
Finally, in order to define the trajectory attractor $\Cal A_{tr}(0)$,
we need to specify the class of "bounded" sets which will be attracted
by this attractor. It is well-known that the topology on the space of
initial data for problem \eqref{0.KV} (with $\alpha=0$, e.g. $\bar u(0)\in H$)
is usually not appropriate for defining bounded sets of trajectories
and the topology on the set of trajectories
 ($\Theta_+$ or its uniformly local analogues) should be used instead,
 see \cite{Ch-V-book,MirZel} and references therein. However, in our
 case such a choice is also insufficient since it may lead to the
 situation when the $\omega$-limit set of $\Cal K^+_0$ will be outside
 $\Cal K_0^+$, so a bit more accuracy is necessary. We overcome this problem,
 following \cite{Zel} (see also \cite{CVZ}) by introducing the so-called
 $M$-functional
 $$
  M_{\bar u}(t):=\inf_{\bar u_{\alpha_n}\rightharpoondown \bar u}\liminf_{n\to\infty}\|\bar u_{\alpha_n}(t)\|^2_{\Cal H_{\alpha_n}},
$$
where the external $\inf$ is taken with respect to all sequences of
solutions of \eqref{0.KV} which converge to a given $\bar u\in\Cal K_0^+$,
see section \S\ref{S3} for more details.  Note that for this quantity,
we have only that  $M_{\bar u}(0)\ge \|\bar u(0)\|_{\Cal H_0}^2$, but no
upper bounds are a priori available. In particular, the value $M_{\bar u}(0)$
depends not only on $\bar u(0)$, but on the whole trajectory $\bar u(t)$, $t\ge0$.
Then, we call a set $B\subset \Cal K_0^+$ bounded if
$$
\sup_{\bar u\in B}M_{\bar u}(0)<\infty.
$$
We are now ready to state the main result of the paper.
\begin{theorem} Let $g\in L^2(\Bbb T)$ and let $\gamma>0$ be fixed.
Then, the limit damped Euler equation (which corresponds to \eqref{0.KV}
with $\alpha=0$) possesses a trajectory attractor $\Cal A_{tr}(0)$ in
the above described sense which attracts bounded sets of $\Cal K_0^+$
in the topology of $\Theta_+$. Moreover, the family of trajectory
attractors $\Cal A_{tr}(\alpha)$, $\alpha>0$   of Bardina--Euler
equations \eqref{0.KV} converges as $\alpha\to0$ to the limit
attractor $\Cal A_{tr}(0)$ in the sense of upper semicontinuity in the space $\Theta_+$.
\end{theorem}
Note that in the case $\alpha>0$, we have the global well-posedness
for problem \eqref{0.KV} and even the existence of {\it global} attractors
$\Cal A(\alpha)$ in the strong topology of $\Cal H_\alpha$, so the
construction of the corresponding trajectory attractors is straightforward.
\par
The paper is organized as follows. In section \S\ref{S01}, we briefly
discuss the known facts on the well-posedness of Bardina--Euler equations on
the 3D torus, state the basic dissipative estimates as well as the results
on the existence of global ($\Cal A(\alpha)$) and trajectory
($\Cal A_{tr}(\alpha)$) attractors for the case $\alpha>0$.
\par
In section \S\ref{S1}, we recall the basic facts about dissipative
solutions adapted to the case of the damped Euler equations and discuss
the existence of dissipative solutions for this equation as well as
their standard properties.
 \par
 In section \S\ref{S3}, we construct the trajectory dynamical system $(\Cal K_0^+,T(h))$ associated with the limit damped Euler equations, introduce the $M$-functional and study its basic properties.
 \par
 Finally, the proof of the main theorem is given in section \S\ref{S4}.

\section{Preliminaries}\label{S01}  We study the following 3D damped
Bardina--Euler system:
%$$
\begin{equation}\label{1.DBE}
\begin{cases}
\Dt \bar u-\alpha\Dt\Dx\bar u+(\bar u,\Nx)\bar u+\Nx p+\gamma (\bar u-\alpha\Dx\bar u)=g,\\
 \divv\bar u=0,\ \ \bar u\big|_{t=0}=\bar u_0
\end{cases}
\end{equation}
%$$
on the 3D torus $\Bbb T:=(-\pi,\pi)^3$.
Here $\bar u=(\bar u_1,\bar u_2,\bar u_3)$ is an unknown velocity field,
$p$ is an unknown pressure, $\Dx$ is the Laplacian with
respect to the variable $x\in\Bbb T$ and $g$ are the given external forces.
We assume that the parameters $\gamma>0$ (Ekman damping)
and $\alpha>0$ (Bardina regularization) are given and will study the damped
Euler limit $\alpha\to0$. It is also natural to assume that
%$$
\begin{equation}\label{2.g}
g\in L^2(\Bbb T)
\end{equation}
%$$
and take the initial data
%$$
\begin{equation}\label{3.h}
\bar u_0\in \Cal H_\alpha:=\{v\in[H^1(\Bbb T)]^3,\ :\ \divv v=0\}
\end{equation}
%$$
endowed with the norm
%$$
\begin{equation}
\|v\|^2_{\Cal H_\alpha}:=\|v\|^2_{L^2}+\alpha\|\Nx v\|^2_{L^2}.
\end{equation}
%$$
The following result concerning the well-posedness of problem \eqref{1.DBE} is straightforward, see e.g. \cite{Titi-Varga,Lay,IKZ}.
\begin{theorem}\label{Th1.ex} Let $g\in L^2(\Bbb T)$ and let $\alpha,\gamma>0$.
Then, for every $\bar u_0\in\Cal H_\alpha$, problem \eqref{1.DBE} possesses a unique solution
 $$
 \bar u\in C^1(0,T;\Cal H_\alpha), \ T>0
 $$
 and this solution satisfies the following dissipative estimate
 %$$
 \begin{equation}\label{1.dis}
 \|\bar u(t)\|_{\Cal H_\alpha}^2\le \|\bar u_0\|_{\Cal H_\alpha}^2e^{-\gamma t}+\frac1{\gamma^2}\|g\|^2_{L^2}.
 \end{equation}
 %$$
\end{theorem}
Thus, problem \eqref{1.DBE} defines a dissipative solution semigroup $S_\alpha(t)$ in the phase space $\Cal H_\alpha$ via:
%$$
\begin{equation}\label{1.sem}
S_\alpha(t):\Cal H_\alpha\to\Cal H_\alpha,\ \ \ S_\alpha(t)\bar u_0:=\bar u(t),
\end{equation}
%$$
where $\bar u(t)$ solves \eqref{1.DBE} with the initial data $\bar u(0)=\bar u_0$.
\par
Moreover, as proved e. g., in \cite{IKZ}, this semigroup possesses a global attractor $\Cal A_\alpha$ in the space $\Cal H_\alpha$ for every $\alpha>0$. The latter means that
\par
1. $\Cal A_\alpha$ is compact in $\Cal H_\alpha$;
\par
2. It is strictly invariant: $S_\alpha(t)\Cal A_\alpha=\Cal A_\alpha$ for all $t>0$;
\par
3. It attracts the images of all bounded sets of $\Cal H_\alpha$, i.e., for any bounded set $B\subset \Cal H_\alpha$ and any neighbourhood $\Cal O(\Cal A_\alpha)$, there is time $T=T(B,\Cal O)$ such that
$$
S_\alpha(t)B\subset\Cal O(\Cal A_\alpha), \ \ t\ge T.
$$
We also recall that the attractor $\Cal A_\alpha$  is generated by all bounded complete trajectories of problem \eqref{1.DBE}:
%$$
\begin{equation}
\Cal A_\alpha=\Cal K_\alpha\big|_{t=0},
\end{equation}
%$$
where $\Cal K_\alpha:=\{\bar u\in C_b(\R,\Cal H_\alpha),\, \bar u(t) \text{ solves \eqref{1.DBE} for all }t\in\R\}$.
\par
Our aim is to study the limit $\alpha\to0$. However, in
 contrast to the case of \eqref{1.DBE}, we do not know whether
 or not the solution of the limit damped Euler equations:
%$$
\begin{equation}\label{1.DE}
\begin{cases}
\Dt\bar u+(\bar u,\Nx)\bar u+\Nx p+\gamma\bar u=g,\\
\divv \bar u=0,\ \ \bar u\big|_{t=0}=\bar u_0
\end{cases}
\end{equation}
%$$
is unique (and even do not have enough regularity to understand this solution in the sense of distributions, see section \S\ref{S1}). For this reason, we briefly recall below the so-called trajectory approach developed in \cite{Ch-V-book} which allows us to overcome the potential non-uniqueness.
\begin{definition} Let $\Cal K_\alpha^+\subset C^1_{loc}(\R_+,\Cal H_\alpha)$ be
the set of all global solutions $\bar u\in C^1_{loc}(\R_+,\Cal H_\alpha)$ of
problem \eqref{1.DBE} which correspond to all possible initial
data $\bar u_0\in\Cal H_\alpha$. This set is called the trajectory phase
space associated with equation \eqref{1.DBE}.
Obviously, the translation semigroup
$$
(T(h)\bar u)(t):=\bar u(t+h),\ \ t,h\ge0
$$
acts on $\Cal K^+_\alpha$:
$$
T(h):\Cal K^+_\alpha\to\Cal K^+_\alpha,\ \ h\ge0,
$$
so the dynamical system $(T(h),\Cal K_\alpha^+)$ is well-defined.
We endow the set $\Cal K_\alpha^+$ with the topology of $C_{loc}(\R_+,\Cal H_\alpha)$
and will refer to the constructed dynamical system $(T(h),\Cal K_\alpha^+)$ as the
trajectory dynamical system associated with equation \eqref{1.DBE}.
\end{definition}
\begin{remark} According to the above constructions,
equation \eqref{1.DBE} generates two dynamical systems (DS): one of them
is the classical DS $(S_\alpha(t),\Cal H_\alpha)$ which is defined
on a usual phase space $\Cal H_\alpha$ of the problem and the second
one is the trajectory DS $(T(t),\Cal K_\alpha^+)$ which is generated by
the translations of the corresponding trajectories.  In the case where
the uniqueness is proved, these two approaches are equivalent. Indeed,
as not difficult to see, the solution map
$$
S:\Cal H_\alpha\to\Cal K_\alpha^+,\ \ S\bar u_0:=\bar u(\cdot)
$$
is a homeomorphism and also maps bounded sets into bounded sets (if we define bounded sets in $\Cal K_\alpha^+$ using the embedding $\Cal K_\alpha^+\subset C_{loc}(\R_+,\Cal H_\alpha)$). Moreover,
%$$
\begin{equation}
T(t)=S\circ S_\alpha(t)\circ S^{-1},\ \ \text{on}\ \ \Cal K_\alpha^+.
\end{equation}
%$$
Thus, the trajectory dynamical system $(T(t),\Cal K_\alpha^+)$ also possesses a global attractor $\Cal A_{tr}(\alpha)$ which is called the trajectory attractor of equation \eqref{1.DBE} and has the following structure:
%$$
\begin{equation}
\Cal A_{tr}(\alpha)=S\Cal A_\alpha=\Cal K_\alpha\big|_{t\ge0},
\end{equation}
%$$
see \cite{Ch-V-book,MirZel} for more details. The advantage of the trajectory approach is that the uniqueness of a solution is {\it not necessary} for constructing the trajectory dynamical system and this will allow us to construct and study the limit attractor $\Cal A_{tr}(0)$ of the damped Euler equations.
\end{remark}

\section{Preliminaries II. The case $\alpha=0$: dissipative solutions}\label{S1}
In this section, we recall the construction of the so-called
{\it dissipative} solutions for the limit damped Euler equations \eqref{1.DE}
 introduced in \cite{PLio}. We remind that, according to estimate \eqref{1.dis},
we are able to control the $L^\infty(\R_+,L^2(\Bbb T))$-norm of the
approximate solutions $\bar u$ of \eqref{1.DBE} as $\alpha\to0$,
so it is natural to consider the solutions of the limit problem belonging
to this space. We are even able to define distributional solutions
in this space using the identity
$$
((\bar u,\Nx)\bar u,\varphi)=-(\bar u\otimes\bar u:\Nx\varphi),
$$
see \cite{PLio} for more details. However, the obtained regularity $u\in L^\infty(\R_+,\Cal H_0)$ is not sufficient for passing to the limit in the nonlinear term $\bar u\otimes\bar u$, so the existence of distributional solutions for the 3D Euler equation remains an open problem. The above mentioned construction of dissipative solutions allows us to overcome this problem. Namely, let $\varphi$ be a sufficiently smooth divergent free test function and let
%$$
\begin{equation}\label{2.D}
D_\alpha(\phi)=D_\alpha(\phi)(t):=\Dt \varphi-\alpha\Dt\Dx\varphi+\Pi(\varphi,\Nx)\varphi+\gamma(\varphi-\alpha\Dx\varphi)-\Pi g,
\end{equation}
%$$
where $\Pi$ is the Leray projector to divergence free vector fields.
Then, if $\bar u$ solves the approximation problem \eqref{1.DBE}, we have the identity
%$$
\begin{multline}
\Dt(\bar u-\varphi)-\alpha\Dt\Dx(\bar u-\varphi)+\gamma(\bar u-\varphi-\alpha\Dx(\bar u-\varphi))+\\+\Pi[(\bar u,\Nx)\bar u-(\varphi,\Nx)\varphi]+D_\alpha(\varphi)=0.
\end{multline}
%$$
Multiplying this identity by $\bar u-\varphi$ and integrating in space, we arrive at
%$$
\begin{equation}\label{2.dif}
\frac12\frac d{dt}\|\bar u-\varphi\|^2_{\Cal H_\alpha}+\gamma\|\bar u-\varphi\|^2_{\Cal H_\alpha}+((\bar u-\varphi,\Nx)\varphi,\bar u-\varphi)+(D_\alpha(\varphi),\bar u-\varphi)=0
\end{equation}
%$$
Let us define the quantity
%$$
\begin{equation}\label{2.e}
e_\alpha(\varphi)(t):=\sup_{z\in\Cal H_\alpha}\frac{-((z,\Nx)\varphi(t),z)}{\|z\|^2_{\Cal H_\alpha}}.
\end{equation}
%$$
Then equality \eqref{2.dif} implies the desired differential inequality
%$$
\begin{equation}
\frac12\frac d{dt}\|\bar u-\varphi\|^2_{\Cal H_\alpha}+(\gamma-e_\alpha(\varphi)(t))\|\bar u-\varphi\|^2_{\Cal H_\alpha}+(D_\alpha(\varphi),\bar u-\varphi)\le0
\end{equation}
%$$
and integrating this in time we finally arrive at the inequality
%$$
\begin{multline}\label{2.sol-dis}
\|\bar u(t)-\varphi(t)\|^2_{\Cal H_\alpha}\le \|\bar u(0)-\varphi(0)\|^2_{\Cal H_\alpha}e^{-2\int_0^t(\gamma-e_\alpha(\varphi)(s))\,ds}-\\-
2\int_0^te^{-2\int_\tau^t(\gamma-e_\alpha(\phi)(s))\,ds}(D_\alpha(\varphi)(\tau),\bar u(\tau)-\varphi(\tau))\,d\tau
\end{multline}
%$$
We are now ready to give the definition of dissipative solutions.
\begin{definition} Let $\alpha\ge0$ and let $\bar u_0\in\Cal H_\alpha$.
Then the function
%$$
\begin{equation}\label{2.spa}
\bar u\in L^\infty_{loc}(\R_+,\Cal H_\alpha)\cap C_{w,loc}(\R_+,\Cal H_\alpha)
\end{equation}
%$$
is a {\it dissipative} solution of problem \eqref{1.DBE} (or \eqref{1.DE} if $\alpha=0$) if, for every smooth divergent free test function $\varphi$ and every $t\ge0$,  inequality \eqref{2.sol-dis} holds.
\end{definition}
The next proposition shows that the concept of dissipative solutions gives
nothing new on the level of the Bardina approximations.
\begin{proposition} Let $\alpha>0$ and let $\bar u_0\in\Cal H_\alpha$. Then the dissipative solution of \eqref{1.DBE} is unique and coincides with the solution $\bar u(t)$ of this problem constructed in Theorem  \ref{Th1.ex}.
\end{proposition}
\begin{proof} Indeed, let $\tilde u(t)$ be a dissipative solution of \eqref{1.DBE}. Following \cite{PLio}, the idea of the proof is just to take $\varphi=\bar u$. To this end we note that, due to the fact that $\bar u\in C^1_{loc}(\R_+,H^1)$, we have $D_\alpha(\bar u)\in C_{loc}(\R_+,H^{-1})$ and $e_\alpha(\bar u)\in L^1_{loc}(\R_+)$. Thus, inequality \eqref{2.sol-dis} is well-defined for $\varphi=\bar u$ and can be justified by the standard density arguments (approximating $\bar u$ by smooth divergence free functions). Taking finally $\varphi=\bar u$, we see that $D_\alpha(\varphi)=0$ and $\tilde u(t)\equiv\bar u(t)$ for all $t$.
\end{proof}
\begin{remark}\label{Rem2.good} In contrast to the case $\alpha>0$,
the regularity of the distributional solution $\bar u(t)$ of \eqref{1.DE}
is not enough to verify that
$D_0(\bar u)\in L^1_{loc}(\R_+,L^2)$ and $e_0(\bar u)\in L^1_{loc}$,
therefore the above given proof does not work for $\alpha=0$. However,
it gives the so-called weak-strong uniqueness in the class of
dissipative solutions. Namely, if a {\it sufficiently regular}
solution $\bar u(t)$ of damped Euler equations \eqref{1.DE} is given
(e.g., the solution which satisfies the Beale--Kato--Majda criterion),
then this solution is unique in the class of dissipative solutions,
see \cite{PLio} for more details.
\par
It is also not difficult to show that the dissipative solution $\tilde u(t)$ of the limit Euler equations solves \eqref{1.DE} in the sense of distributions if it is regular enough.  Indeed, let us take $\varphi_\eb(t):=\tilde u+\eb\theta$, where $\eb\in\R_+$ is a small parameter and $\theta$ is an arbitrary smooth divergent free function. Then, dividing \eqref{2.sol-dis} by $\eb$ and passing to the limit $\eb\to0$, we arrive at
$$
\int_0^te^{-2\int_\tau^t(\gamma-e_0(\tilde u)(s))\,ds}(D_0(\tilde u)(\tau),\theta(\tau))\,d\tau=0
$$
for every test function $\theta$. This gives $D_0(\tilde u)\equiv0$ and $\tilde u$ is a distributional solution of \eqref{1.DE}. Thus, despite a bit unusual form, the concept of a dissipative solution is a natural and convenient generalization of distributional solutions.
\par
In addition, there is a natural connection between dissipative and distributional solutions. Indeed, as shown in \cite{BarTi} (for the case $\gamma=g=0$, but the general case is analogous), any distributional solution $\tilde u$ of \eqref{1.DE} which satisfies the energy inequality:
$$
\|\tilde u(t)\|^2_{\Cal H_0}\le\|\tilde u(0)\|^2_{\Cal H_0}e^{-2\gamma t}+2\int_0^te^{-2\gamma(t-s)}(g,\tilde u(s))\,ds
$$
is automatically a dissipative solution.
\end{remark}
We conclude this section by establishing the existence of a dissipative solution.

\begin{proposition}\label{Prop2.sol} Let $\bar u_0\in \Cal H_0$. Then
there exists at least one dissipative solution of damped Euler equations \eqref{1.DE}.
\end{proposition}
\begin{proof} Let $\alpha>0$ and let $\bar u_0^\alpha\in \Cal H_\alpha$ be  such that
 $$
 \bar u_0^\alpha\to\bar u_0 \text{ in $\Cal H_0$ and } \|\bar u_0^\alpha\|_{\Cal H_\alpha}\to \|\bar u_0\|_{\Cal H_0} \text{ as $\alpha\to0$}.
 $$
In particular, one can take $\bar u_0^\alpha:=(1-\alpha\Dx)^{-1/2}\bar u_0$. Let
$\bar u_\alpha(t):=S_\alpha(t)\bar u_0^\alpha$ be the corresponding solutions
of damped Bardina--Euler equations \eqref{1.DBE}. Then, according to
Theorem~\eqref{Th1.ex} the functions $\bar u_\alpha$ are uniformly
bounded in $C_b(\R,\Cal H_0)$.%$ $C_b(\R_+,\Cal H_\alpha)\subset C_b(\R.\Cal H_0)$.
 \ Moreover, it is
 not difficult to see that
$$
\|(\bar u_\alpha,\Nx)\bar u_\alpha)\|_{H^{-3}(\Bbb T)}\le C\|\bar u_\alpha\|_{\Cal H_0}^2
$$
and, therefore, $\Dt\bar u_\alpha$ are uniformly bounded in
$C_{loc}(\R_+, H^{-3}(\Bbb T))$. Thus, by Banach--Alaoglu theorem,
we may assume without loss of generality that
$$
\bar u_\alpha\rightharpoondown \tilde u,\ \ \Dt \bar u_\alpha\rightharpoondown \Dt\tilde u\ \ \text{as $\alpha\to0$}
$$
weakly star in $L^\infty_{loc}(\R_+,\Cal H_0)$ and $L^\infty_{loc}(\R_+,H^{-3})$
to some function $\tilde u$. In particular, this convergence implies that
$\tilde u\in C_{w,loc}(\R_+,\Cal H_0)$, $\Dt \tilde u\in L^\infty_{loc}(\R_+,H^{-3})$
and that for every $t\in\R_+$ we have a weak convergence
$\bar u_\alpha(t)\rightharpoondown\tilde u(t)$ in $\Cal H_0$.
\par
It only remains to check that $\tilde u$ is a desired
dissipative solution of \eqref{1.DE}. Let $\varphi$
be a smooth divergent free test function. Then, obviously,
$e_\alpha(\varphi)\to e_0(\varphi)$ strongly in $L^1_{loc}$ and
also $D_\alpha(\varphi)\to D_0(\varphi)$ strongly in $L^1_{loc}(\R_+,L^2)$.
 Writing down the variational inequality \eqref{2.sol-dis} for the dissipative solutions $\bar u_\alpha$ and passing to the limit $\alpha\to0$, we now see that $\tilde u$ satisfies the variational inequality for $\alpha=0$ which finishes the proof of the proposition.
\end{proof}

\section{The  trajectory dynamical system}\label{S3}
In this section, we develop the trajectory approach for dissipative solutions of
limit damped Euler equations \eqref{1.DE}. Following the general approach, we
want to fix the trajectory phase space $\Cal K_+$ as a set of all dissipative
solutions of \eqref{1.DE} which correspond to $\bar u_0\in\Cal H_0$. However,
there is one more problem here, namely, the class of dissipative solutions
is {\it not invariant} with respect to time shifts $T(t)$ since the variational inequality \eqref{2.sol-dis} is not invariant. As a result, the $\omega$-limit set of a subset of $\Cal K_0^+$ will not belong to $\Cal K_0^+$ in general and we will be unable to characterise the constructed attractor in terms of dissipative solutions of the considered problem. In order to overcome this difficulty, we restrict further (following e.g. \cite{Zel}, see also \cite{CVZ,MirZel})
the class of dissipative solutions by considering only those of them which can
be obtained as weak limits of the corresponding solutions of damped Bardina--Euler
equations \eqref{1.DBE} as $\alpha\to0$. Namely,
\begin{definition} The trajectory phase space
%$$
\begin{equation}\label{3.top}
\Cal K_0^+\subset L^\infty_{loc}(\R_+,\Cal H_0)\cap W^{1,\infty}_{loc}(\R_+,H^{-3})
\end{equation}
%$$
is defined as follows: $\tilde u(t)\in\Cal K_0^+$ if and only if there
exists a sequence $\alpha_n\to0$, $\alpha_n>~0$ and a sequence of the initial data $\bar u_0^{\alpha_n}$ such that $\|\bar u_0^{\alpha_n}\|_{\Cal H_{\alpha_n}}$ is uniformly bounded in $\Cal H_{\alpha_n}$ and the corresponding solutions
$\bar u_n(t):=S_{\alpha_n}(t)\bar u_0^{\alpha_n}$ of equations \eqref{1.DBE} converge {\it weakly} in $\Cal H_0$ to $\tilde u(t)$ for every $t\ge0$. Then, obviously, the translation semigroup $T(h)$ acts on $\Cal K_0^+$. We endow the trajectory phase space $\Cal K_0^+$ with the {\it weak star topology} induced by embedding \eqref{3.top} and will refer to $(T(h),\Cal K_0^+)$ as the trajectory dynamical system associated with damped Euler equation \eqref{1.DE}.
\end{definition}
\begin{remark} Arguing as in the proof of Theorem \ref{Th1.ex}, we see that the set $\Cal K_0^+$ is not empty and possesses a sufficiently large number of trajectories, namely, for any $\bar u_0\in\Cal H_0$ there exists at least one element $\tilde u\in \Cal K_0^+$ such that $\tilde u(0)=\bar u_0$ (we may take a dissipative solution constructed in Theorem \ref{Th1.ex} as such a $\tilde u$).
\par
However, in contrast to Theorem \ref{Th1.ex}, we now unable to pass to the limit $\alpha_n\to0$ and get that any $\tilde u\in\Cal K_0^+$ is a dissipative solution. The problem is that we now require only the weak convergence $\bar u_0^{\alpha_n}\to\bar u_0$ and cannot require the convergence of the norms ($\|\bar u_0^{\alpha_n}\|_{\Cal H_{\alpha_n}}\to\|\bar u_0\|_{\Cal H_0}$) since this norm convergence cannot be verified for $t>0$ and would make $\Cal K_0^+$ not translation-invariant. For this reason, the concept of a dissipative solution should be modified. We will do this modification using the so-called $M$-functional introduced in \cite{Zel} for study the supercritical damped wave equations, see also \cite{CVZ}.
\end{remark}
\begin{definition} Let $\tilde u\in\Cal K_0^+$ and let $\varphi$ be a smooth divergent free test function. Define
%$$
\begin{equation}\label{3.M}
M_{\tilde u,\varphi}(t):=\inf_{\bar u_{\alpha_n}\rightharpoondown \tilde u}\liminf_{n\to\infty}\|\bar u_{\alpha_n}(t)-\varphi(t)\|^2_{\Cal H_{\alpha_n}},
\end{equation}
%$$
where the external infimum is taken over all solutions $\bar u_{\alpha_n}$ of
damped Bardina--Euler equations \eqref{1.DBE} such that $\alpha_n\to0$ as $n\to\infty$, $\alpha_n>0$, the sequence $\bar u_{\alpha_n}(0)$ is uniformly bounded in $\Cal H_{\alpha_n}$ and such that
%$$
\begin{equation}
\bar u_{\alpha_n}(t)\rightharpoondown\tilde u(t)
\end{equation}
%$$
weakly in $\Cal H_0$ for every $t\ge0$.
\end{definition}
The next proposition collects some straightforward properties of the introduced $M$-functional.
\begin{proposition} Let $\tilde u\in\Cal K_0^+$ and $\varphi$ be a smooth divergent free function. Then,
\par
1) $\|\tilde u(t)-\varphi(t)\|^2_{\Cal H_0}\le M_{\tilde u,\varphi}(t)$ for all $t\ge0$;
\par
2) $M_{\tilde u,\varphi}(t)=M_{\tilde u,0}(t)+\|\varphi\|^2_{\Cal H_0}-2(\tilde u,\varphi)$ for all $t\ge0$;
\par
3) $M_{T(h)\tilde u,T(h)\varphi}(t)\le M_{\tilde u,\varphi}(t+h)$ for all $t,h\ge0$;
\par
4) The following analogue of variational inequality \eqref{2.sol-dis} holds:
%$$
\begin{multline}\label{3.dis-mod}
M_{\tilde u,\varphi}(t+\kappa)\le M_{\tilde u,\varphi}(t)e^{-2\int_t^{t+\kappa}(\gamma-e_0(\varphi)(s))\,ds}-\\-
2\int_t^{t+\kappa}e^{-2\int_\tau^{t+\kappa}(\gamma-e_0(\varphi)(s))\,ds}(D_0(\varphi)(\tau),\tilde u(\tau)-\varphi(\tau))\,d\tau
\end{multline}
%$$
for all $t,\kappa\ge0$.
\end{proposition}
\begin{proof} Indeed, the first three statements are obvious, so we only
need to check the fourth one. Let $\bar u_{\alpha_n}(t)$ be a sequence
of solutions of Bardina--Euler approximations \eqref{1.DBE} which
converges weakly to $\tilde u(t)$ as $\alpha_n\to0$ and let $\varphi$
be a smooth divergence free test function. Then,
analogously to \eqref{2.sol-dis}, we have the inequality
%$$
\begin{multline}\label{3.lim}
\|\bar u_{\alpha_n}(t+\kappa)-\varphi(t+\kappa)\|_{\Cal H_{\alpha_n}}^2\le \|\bar u_{\alpha_n}(t)-\varphi(t)\|_{\Cal H_{\alpha_n}}^2e^{-2\int_t^{t+\kappa}(\gamma-e_{\alpha_n}(\varphi)(s))\,ds}-\\-
2\int_t^{t+\kappa}e^{-2\int_\tau^{t+\kappa}(\gamma-e_{\alpha_n}(\varphi)(s))\,ds}(D_{\alpha_n}(\varphi)(\tau),\bar u_{\alpha_n}(\tau)-\varphi(\tau))\,d\tau.
\end{multline}
%$$
Thus, we only need to pass to the limit in \eqref{3.lim} in a proper way.
The passage to the limit in the last term in the RHS
is immediate since we have a weak convergence of
$\bar u_{\alpha_n}(t)$ to $\tilde u(t)$ for all $t$.
Taking $\liminf_{\alpha_n\to0}$ followed by
$\inf_{\bar u_{\alpha_n}\to\tilde u}$ from both
sides of \eqref{3.lim}, we get \eqref{3.dis-mod}
and finish the proof of the proposition.
\end{proof}
The next proposition is crucial for our construction of trajectory attractors.
\begin{proposition}\label{Prop3.main} Let $\{\tilde u^l\}_{l\in\Bbb N}\subset\Cal K^+_0$ be such that the sequence $M_{\tilde u^l,0}(0)$ is uniformly bounded and
%$$
\begin{equation}
\tilde u^l\rightharpoondown \tilde u \ \ \text{ weakly star in }\ \ L^\infty_{loc}(\R_+,\Cal H_0).
\end{equation}
%$$
Then $\tilde u\in\Cal K_0^+$ and
%$$
\begin{equation}\label{3.w-cont}
M_{\tilde u,0}(t)\le\liminf_{l\to\infty}M_{\tilde u^l,0}(t)
\end{equation}
%$$
for all $t\ge0$.
\end{proposition}
\begin{proof} Taking $\varphi\equiv0$ in \eqref{3.dis-mod} and using the assumed uniform boundedness of $M_{\tilde u^l,0}(0)$, we get that
%$$
\begin{equation}
\|\tilde u^l(t)\|_{\Cal H_0}^2\le M_{\tilde u^l,0}(t)\le C
\end{equation}
%$$
for all $t\ge0$ and $l\in\Bbb N$.  Moreover, by the definition of the
$M$-functional, for any $l\in\Bbb N$, there exist a sequence
$\alpha_{n,l}>0$, $\lim_{n\to\infty}\alpha_{n,l}=0$ and a sequence
of initial data $\bar u_0^{\alpha_{n,l}}\in\Cal H_{\alpha_{n,l}}$
such that  the corresponding solutions
$\bar u_{\alpha_{n,l}}(t):=S(t)\bar u_0^{\alpha_{n,l}}$ of damped Bardina--Euler
equations \eqref{1.DBE} converge weakly to $\tilde u^l(t)$ in $\Cal H_0$
for any $t\ge0$. Moreover, since $M_{\tilde u^l,0}(0)$ are uniformly bounded,
we may assume without loss of generality that
%$$
\begin{equation}
\|\bar u_0^{\alpha_{n,l}}\|_{\Cal H_{\alpha_{n,l}}}\le C,
\end{equation}
%$$
where $C$ is independent of $n$ and $l$. Arguing now as in Proposition \ref{Prop2.sol}, we see that
%$$
\begin{equation}
\|\bar u_{\alpha_{n,l}}(t)\|_{\Cal H_{\alpha_{n,l}}}+\|\Dt \bar u_{\alpha_{n,l}}(t)\|_{H^{-3}}\le C,\ \ t\ge0.
\end{equation}
%$$
In other words, the sequence $\bar u_{\alpha_{n,l}}$ is uniformly bounded
in the space $L^\infty(\R_+,\Cal H_0)\cap W^{1,\infty}(\R_+,H^{-3})$ and the
Banach--Alaoglu theorem now guarantees that
$$
\tilde u\in L^\infty(\R_+,\Cal H_0)\cap W^{1,\infty}(\R_+,H^{-3}).
$$
To verify that $\tilde u\in\Cal K_0^+$, we only need to construct a
sequence of solutions of the approximating damped Bardina--Euler
equations which is convergent to $\tilde u$. To this end, we note that
any bounded set in $L^\infty_{loc}(\R_+,\Cal H_0)\cap W^{1,\infty}_{loc}(\R_+,H^{-3})$
is metrizable in weak-star topology. Let $B$ be a bounded set of this space which
contains all trajectories $\bar u_{\alpha_{n,l}}$ and $\tilde u$ and
let $d(\cdot,\cdot)$ be such a metric.
Then, we have the convergence $\bar u_{\alpha_{n,l}}\to\tilde u^l$ as $n\to\infty$
in a metric space $(B,d)$ for every $l\in\Bbb N$ as well as the
convergence $\tilde u^l\to \tilde u$ as $l\to\infty$. Thus,
there exists a sequence $\bar u_{\alpha_{n_l,l}}\to\tilde u$
in $(B,d)$ and the approximating sequence is constructed.
This proves the inclusion $\tilde u\in\Cal K_0^+$.
\par
The proof of the weak lower semicontinuity \eqref{3.w-cont} is straightforward and we left it to the reader.
\end{proof}
\begin{remark} Note that the properties
of "solutions" $\tilde u\in\Cal K_0^+$
are more delicate than the analogous properties of
dissipative solutions discussed in Remark \ref{Rem2.good}.
Indeed, we are still able to prove that any sufficiently regular
distributional solution is automatically in $\Cal K_0^+$. However,
by  definition, $\Cal K_0^+$ contains only those solutions which can
be obtained as a limit of regularized solutions of damped Bardina--Euler
equation and it is not clear whether or not any dissipative solutions can be
obtained in such a way. Moreover, keeping in mind the oscillatory shear flow
solutions of Euler equations constructed by Di Perna and Majda, see \cite{DiM}
as well as the result of Proposition \ref{Prop3.main}, one may expect that there
are {\it regular} trajectories $\tilde u\in\Cal K_0^+$ which do not satisfy the
Euler equation \eqref{1.DE} in the sense of distributions. Actually, the main
thing which we know is that $\Cal K_0^+$ is not empty and, for every
$\tilde u_0\in\Cal H_0$ there is a dissipative solution $\tilde u\in\Cal K_0^+$
such that $\tilde u(0)=\tilde u_0$, Note that this situation is somehow standard
for the theory of trajectory attractors, see \cite{Ch-V-book} and  the  references therein.
\end{remark}

\section{The trajectory attractor}\label{S4}
 In this section, we construct a weak global attractor of the trajectory dynamical system $(T(h),\Cal K_0^+)$ associated with the damped Euler equation \eqref{1.DE}. Recall that we have already introduced the topology on the trajectory phase space $\Cal K_0^+$ as the topology induced by the embedding
 $$
 \Cal K_0^+\subset \Theta_+:=L^\infty_{w^*,loc}(\R_+,\Cal H_0)\cap W^{1,\infty}_{w^*,loc}(\R_+,H^{-3}).
 $$
However, this topology is not convenient for defining the bornology (= the class of bounded sets) which is necessary for the attractor theory. Instead, following \cite{Zel} (see also \cite{CVZ,MirZel}) we give the following definitions.
\begin{definition}\label{Def4.attr} A set $B\subset\Cal K_0^+$ is bounded in $\Cal K_0^+$ if
$$
\sup_{\tilde u\in B}M_{\tilde u,0}(0)\le C_B<\infty.
$$
A set $\Cal A_{tr}(0)\subset\Cal K_0^+$ is a weak global attractor of the trajectory dynamical system $(T(h),\Cal K_0^+)$ (= a (weak) trajectory attractor of equation \eqref{1.DE}) if
\par
1. $\Cal A_{tr}(0)$ is a compact set in $\Cal K_0^+$;
\par
2. $\Cal A_{tr}$ is strictly invariant: $T(h)\Cal A_{tr}(0)=\Cal A_{tr}(0)$ for all $h\ge0$;
\par
3. $\Cal A_{tr}(0)$ attracts the images of all bounded sets of $\Cal K_0^+$, i.e., for every bounded set $B\subset \Cal K_0^+$ and every neighbourhood $\Cal O(\Cal A_{tr}(0))$ of $\Cal K_0^+$ in $\Theta_+$, there exists $T=T(B,\Cal O)$ such that
$$
T(t)B\subset\Cal O(\Cal A_{tr}(0) \text{ for all } t\ge T.
$$
\end{definition}
The following theorem can be considered as the main result of the paper.

\begin{theorem} Let $g\in \Cal H_0$. Then the damped Euler equation \eqref{1.DE} possesses a weak trajectory attractor $\Cal A_{tr}(0)$ in the sense of Definition \ref{Def4.attr}. Moreover, this attractor is generated by complete bounded trajectories of this equation:
%$$
\begin{equation}\label{4.rep}
\Cal A_{tr}(0)=\Cal K_0\big|_{t\ge0},
\end{equation}
%$$
where $\tilde u\in\Cal K_0\subset L^\infty(\R,\Cal H_0)\cap W^{1,\infty}(\R,H^{-3})$
if and only if there exist sequences $t_n\to-\infty$, $\alpha_n\to0$ and
$\bar u_0^{\alpha_n}\in\Cal H_{\alpha_n}$ such that the norms
$\|\bar u_0^{\alpha_n}\|_{\Cal H_{\alpha_n}}$ are uniformly bounded as
$n\to\infty$ and the corresponding solutions
$u_{\alpha_n}(t):=S(t+t_n)\bar u_0^{\alpha_n}$, $t\ge-t_n$ of the corresponding
damped Bardina--Euler equations \eqref{1.DBE} converge weakly to the
function $\tilde u$:
$$
\bar u_{\alpha_n}(t)\rightharpoondown \tilde u(t)\ \ \text{in $\Cal H_0$ for all }\ t\in \R.
$$
\end{theorem}
\begin{proof} According to the abstract theorem on
 existence of global attractors, see \cite{B-V,Ch-V-book,T} and the
 references therein, we need to verify that
\par
1. There exists a bounded absorbing set $\Cal B\subset \Cal K_0^+$ of the translation semigroup $T(h)$ which is a metrizable compact in $\Cal K_0^+$;
\par
2. The operators $T(h)$ are continuous on $\Cal B$ for every fixed $h\ge0$.
\par
The representation formula \eqref{4.rep} is then a standard corollary of this theorem.
\par
Let us construct an absorbing set $\Cal B$ for the semigroup $T(h)$ with the desired properties. Indeed, taking $\varphi\equiv0$ in \eqref{3.dis-mod}, we have
%$$
\begin{multline}
M_{\tilde u,0}(t+\kappa)\le M_{\tilde u,0}(t)e^{-2\gamma\kappa}-2\int_t^{t+\kappa}e^{-2\gamma(t+\kappa-\tau)}(g,\tilde u(\tau))\,d\tau\le \\\le
M_{\tilde u,0}(t)e^{-2\gamma\kappa}+2\int_t^{t+\kappa}e^{-2\gamma(t+\kappa-\tau)}\|g\|_{\Cal H_0}\sqrt{M_{\tilde u,0}(\tau)}\,d\tau
\end{multline}
%$$
and the Gronwall inequality gives that
%$$
\begin{equation}
M_{\tilde u,0}(t+\kappa)\le M_{\tilde u,0}(t)e^{-2\gamma\kappa}+\frac1{\gamma^2}\|g\|^2_{\Cal H_0}
\end{equation}
%$$
and, therefore,
%$$
\begin{equation}\label{4,abs}
\Cal B:=\{\tilde u\in\Cal K_0^+\,: \ M_{\tilde u,0}(0)\le \frac2{\gamma^2}\|g\|_{\Cal H_0}^2\}
\end{equation}
%$$
is an absorbing set for the semigroup $T(h)$ in $\Cal K_0^+$. By the Banach-Alaoglu theorem, $\Cal B$ is precompact and metrizable in $\Theta_+$ and due to Proposition \ref{Prop3.main} $\Cal B$ is closed in $\Cal K_0^+$. Thus, $\Cal B$ is a metric compact in $\Cal K_0^+$. The continuity of $T(h)$ is obvious since the shift semigroup $T(h)$ is continuous on $\Theta^+$.
\par
Thus, all the assumptions of the abstract attractors existence theorem are verified and the theorem is proved.
\end{proof}
We conclude the section by proving the upper semicontinuity of the attractors $\Cal A_{tr}(\alpha)$ at $\alpha=0$.
\begin{corollary} Let $g\in\Cal H_0$ and let $\Cal A_{tr}(\alpha)$, $\alpha>0$
and $\Cal A_{tr}(0)$ be the trajectory attractors of the damped Bardina--Euler
equations \eqref{1.DBE} and the limit damped Euler equations \eqref{1.DE},
respectively. Then, for every neighbourhood of $\Cal O(\Cal A_{tr}(0))$ of
the limit attractor $\Cal A_{tr}(0)$ in the topology of $\Theta^+$, there
exists $\alpha(\Cal O)>0$ such that
%$$
\begin{equation}
\Cal A_{tr}(\alpha)\subset\Cal O(\Cal A_{tr}(0)) \ \ \text{ for all}\ \alpha\le\alpha(\Cal O).
\end{equation}
%$$
\end{corollary}
\begin{proof} The statement of the corollary is almost tautological. Indeed, as we have verified before, the sets $\Cal K_{\alpha}$, $\alpha>0$ of complete trajectories are
 bounded in the space $L^\infty(\R,\Cal H_\alpha)\cap W^{1,\infty}(\R,H^{-3})$ uniformly with respect to $\alpha\to0$ and by Banach-Alaoglu theorem, for any sequences $\alpha_n\to0$ and $\bar u_{\alpha_n}\in\Cal K_{\alpha_n}$ there is a subsequence (which we denote by $\bar u_{\alpha_n}$ again) which is convergent in the space
 $$
 \Theta:=L^\infty_{w^*,loc}(\R,\Cal H_0)\cap W^{1,\infty}_{w^*,loc}(\R,H^{-3})
 $$
 to some function $\tilde u\in\Theta$. To verify the upper semicontinuity, it is sufficient to prove that we necessarily have $\tilde u\in\Cal K_0$. This can be done exactly as in the proof of Proposition \ref{Prop3.main}. Thus, the corollary is proved.
\end{proof}

\end{document}